\documentclass[12pt]{amsart}
\usepackage{amssymb,amsmath,amsfonts,latexsym}
\usepackage{bm}
\usepackage[all,cmtip]{xy}
\usepackage{amscd}
\usepackage{multirow,bigdelim}

\usepackage{xcolor}

\newcommand{\newsection}[1]{\setcounter{equation}{0} \section{#1}}
\newcommand{\bea}{\begin{eqnarray}}
\newcommand{\eea}{\end{eqnarray}}

\newcommand{\vp}{\varphi}

\newcommand{\cla}{\mathcal{A}}
\newcommand{\clb}{\mathcal{B}}

\newcommand{\cld}{\mathcal{D}}
\newcommand{\cle}{\mathcal{E}}
\newcommand{\clf}{\mathcal{F}}
\newcommand{\clg}{\mathcal{G}}
\newcommand{\clh}{\mathcal{H}}
\newcommand{\clk}{\mathcal{K}}

\newcommand{\clm}{\mathcal{M}}

\newcommand{\clq}{\mathcal{Q}}
\newcommand{\clr}{\mathcal{R}}
\newcommand{\cls}{\mathcal{S}}

\newcommand{\z}{\bm{z}}
\newcommand{\w}{\bm{w}}

\newcommand{\D}{\mathbb{D}}

\newcommand{\raro}{\rightarrow}

\def\textmatrix#1&#2\\#3&#4\\{\bigl({#1 \atop #3}\ {#2 \atop #4}\bigr)}
\def\dispmatrix#1&#2\\#3&#4\\{\left({#1 \atop #3}\ {#2 \atop #4}\right)}
\newcommand{\be}{\begin{equation}}
\newcommand{\ee}{\end{equation}}
\newcommand{\ben}{\begin{eqnarray*}}
\newcommand{\een}{\end{eqnarray*}}

\newcommand{\NI}{\noindent}

\newcommand{\bi}{\begin{itemize}}
\newcommand{\ei}{\end{itemize}}

\theoremstyle{definition}

\theoremstyle{plain}

\newtheorem{thm}{Theorem}[section]

\theoremstyle{definition}

\newtheorem{rem}[thm]{Remark}
\newtheorem{ex}[thm]{Example}

\numberwithin{equation}{section}

\def\onegroup{\hskip\arraycolsep\underbrace{0 \hskip\arraycolsep\dots \hskip\arraycolsep 0}_{(j+1)}}
\let\phi=\varphi

\begin{document}

\title[Schur functions and inner functions on the bidisc]{Schur functions and inner functions on the bidisc}

\author[Debnath]{Ramlal Debnath}
\address{Indian Statistical Institute, Statistics and Mathematics Unit, 8th Mile, Mysore Road, Bangalore, 560059, India}
\email{ramlal\_rs@isibang.ac.in, ramlaldebnath100@gmail.com}

\author[Sarkar]{Jaydeb Sarkar}
\address{Indian Statistical Institute, Statistics and Mathematics Unit, 8th Mile, Mysore Road, Bangalore, 560059, India}
\email{jay@isibang.ac.in, jaydeb@gmail.com}

%\today



\begin{abstract}
We study representations of inner functions on the bidisc from a fractional linear transformation point of view. We provide sufficient conditions, in terms of colligation matrices, for the existence of two-variable inner functions. Here the sufficient conditions are not necessary in general, and we prove a weak converse for rational inner functions that admit one variable factorization.

\noindent We present a classification of de Branges-Rovnyak kernels on the bidisc (which equally works in the setting of polydisc and the open unit ball of $\mathbb{C}^n$, $n \geq 1$). We also classify, in terms of Agler kernels, two-variable Schur functions that admit one variable factor.
\end{abstract}

\maketitle

\newsection{Introduction}

Let $\D^n = \{\z = (z_1, \ldots, z_n) \in \mathbb{C}^n: |z_1|, \ldots, |z_n| <1\}$ denote the open unit polydisc and $H^\infty(\D^n)$ denote the Banach algebra of all bounded analytic functions on $\D^n$ with the uniform norm
\[
\|\vp\|_\infty := \sup \{|\vp(\z)|: \z \in \D^n\} \qquad (\vp \in H^\infty(\D^n)).
\]
A function $\vp \in H^\infty(\D^n)$ is said to be \textit{Schur function} if $\|\vp\|_\infty \leq 1$. We denote by $\cls(\mathbb{D}^n)$ the set of Schur functions defined on $\D^n$. A function $\vp \in \cls(\mathbb{D}^n)$ is said to be \textit{inner} if
\[
\lim\limits_{r\nearrow 1}|\vp(re^{it_1},\cdots, re^{it_n})|=|\vp(e^{it_1},\cdots, e^{it_n})|=1,
\]
almost everywhere on the distinguished boundary $\mathbb{T}^n$ of $\D^n$. For example, every rational inner function $\vp \in \cls(\D^n)$ has the form
\[
\vp(\z)= M \frac{\overline{p(\frac{1}{\bar{z}_1}, \ldots, \frac{1}{\bar{z}_n} )}}{p(\z)},
\]
where $M$ is a monomial and $p$ is a polynomial with no zeros in $\D^n$ (see Rudin \cite[Theorem 5.2.5]{Rud}).

The principle aim of this paper is threefold: (1) Representations of inner functions in $\cls(\D^2)$ in terms of isometric colligation operators (a certain class of $2 \times 2$ block operator matrices). (2) Classification of de Branges-Rovnyak kernels on $\D$ (which equally works in the setting of $\D^n$ and the open unit ball in $\mathbb{C}^n$, $n \geq 1$). (3) Classification, in terms of Agler kernels, of Schur functions in $\cls(\D^2)$ that admit one variable factor.

To further elaborate on the main contribution of this paper, we recall some classical results. First let us recall the fractional linear transformation representations of Schur functions on $\D$ \cite{AM, F}. Let $\vp : \D \raro \mathbb{C}$ be a function. Then $\vp \in \cls(\D)$ if and only if there exist a Hilbert space $\clh$ (known as \textit{state space}) and a $2 \times 2$ block operator matrix (known as \textit{colligation matrix/operator})
\[
V = \begin{bmatrix} a & B \\ C & D \end{bmatrix} : \mathbb{C} \oplus \mathcal{H} \rightarrow \mathbb{C} \oplus \mathcal{H},
\]
such that $V$ is an isometry(/co-isometry/unitary/contraction) and $\vp = \tau_V$, where
\[
\tau_V(z) = a + z B (I_{\mathcal{H}} - z D)^{-1} C \quad \quad (z \in \mathbb{D}).
\]
We call $\tau_V$ the \textit{transfer function} or the \textit{realization function} corresponding to the colligation operator $V$. Moreover (see \cite[Theorem 7.10, page 110, and Theorem 10.1, page 122]{F}):

\begin{thm}\label{thm: classical inner}
Let $\vp\in \cls(\D)$. Then $\vp$ is inner if and only if $\vp = \tau_V$ for some isometric colligation
\[
V=\begin{bmatrix}
a&B\\
C&D
\end{bmatrix} \in \clb(\mathbb{C}\oplus\mathcal{H}),
\]
with $D\in C_{0\cdot}$.
\end{thm}

Here $C_{0\cdot}$ denotes the set of all contractions $T$ (on Hilbert spaces) such that $T^n \raro0$ in the strong operator topology, and $\clb(\clh, \clk)$ (or simply $\clb(\clh)$ if $\clh = \clk$) denotes the Banach space of bounded linear operators from a Hilbert space $\clh$ into a Hilbert space $\clk$. We present a simpler proof of the necessary part of the above theorem at the end of this section.

Our first aim in this paper is to provide similar sufficient (as well as necessary, for reducible rational functions) conditions for a function in $\cls(\D^2)$ to be inner. Our presentation here, needless to say, is based on Agler's realization formula and Agler kernels for functions in $\cls(\D^2)$ \cite{JA1}. Let us briefly recall the definition of kernel functions. Let $\cle$ be a Hilbert space, and let $\Omega$ be a domain in $\mathbb{C}^n$. A function $K: \Omega \times \Omega \rightarrow \clb(\cle)$  is called a \textit{kernel} (denoted by $K \geq 0$) if
\[
\sum\limits_{i,j=1}^m\langle K(\z_i,\z_j)\eta_j,\eta_i \rangle_{\cle} \geq 0,
\]
for all $\{\z_1,\dots, \z_m\}\subseteq \Omega$, $\{\eta_1,\dots, \eta_m\} \subseteq \cle$  and $m \geq 1$ (see the monograph by Paulsen and Raghupathi \cite{RP}, or Szafraniec \cite{Sz} for a rapid introduction to reproducing kernels).

\begin{thm}[Agler]\label{thm-Agler D2}
Let $\vp: \D^2 \raro \mathbb{C}$ be a function. Then the following are equivalent:

(i) $\vp \in \cls (\D^2)$.

(ii) There exist Hilbert spaces $\clh_1$ and $\clh_2$ and a unitary/isometric/co-isometric colligation operator
\[
V = \begin{bmatrix} a & B \\ C & D \end{bmatrix} \in \clb(\mathbb{C} \oplus (\mathcal{H}_1\oplus\mathcal{H}_2)),
\]
such that $\vp = \tau_V$, where
\[
\tau_V(\z) = a + B (I_{\clh_1 \oplus \clh_2} -E_{\clh_1 \oplus \clh_2}(\z) D)^{-1} E_{\clh_1 \oplus \clh_2}(\z) C,
\]
and $E_{\clh_1 \oplus \clh_2}(\z) = z_1 I_{\clh_1} \oplus z_2 I_{\clh_2}$ for all $\z \in \D^2$.

(iii) There exist kernels $\{K_1, K_2\}$ such that
\[
1 - \vp(\z) \overline{\vp(\w)} = (1 - z_1 \bar{w}_1) K_1(\z, \w) + (1 - z_2 \bar{w}_2) K_2(\z, \w) \qquad (\z, \w \in \D^2).
\]
\end{thm}

The kernels $\{K_1, K_2\}$ in (ii) are known as \textit{Agler kernels} of $\vp$, and the identity in (iii) is known as the \textit{Agler decomposition} of $\vp$ corresponding to the Agler kernels $\{K_1, K_2\}$ \cite[Section 3.1]{BB}. We also call $\clh_1 \oplus \clh_2$ in (ii) as a \textit{state space} of $\vp$.

Agler kernels were essentially introduced by Jim Agler
\cite{Agler preprint} (also see \cite{JA1} and \cite{AM1}) in his study of Nevanlinna-Pick interpolation in the setting of bidisc. Agler kernels also play an important role (cf. \cite{BL1, GK,SYY}) in the delicate structure of shift-invariant subspaces of the Hardy space over the bidisc \cite{GW}. Moreover, in the case of $\mathbb{D}^n$, $n \geq 3$, Agler kernels are associated with those bounded analytic functions on $\D^n$ that satisfy the multivariable von Neumann inequality (see the discussion following Theorem \ref{cor: de Branges NF}). Needless to say, bounded analytic functions on $\D^n$, $n \geq 3$ satisfying the multivariable von Neumann inequality are of interest. From these points of view, the concept of Agler kernels has become an inseparable part of the modern theory of bounded analytic functions.

We now return to the topic of representations of inner functions in $\cls(\D^2)$. An analog of Theorem \ref{thm: classical inner} for inner functions in $\cls(\D^2)$ seems to be a subtle and unattended problem. Here the main difficulty is to deal with the $2 \times 2$ block operator matrix $D \in \clb(\clh_1 \oplus \clh_2)$, or more specifically, with the resolvent part of $\tau_V$ which involves inverse of $2 \times 2$ block operator matrix. Instead, in Theorem \ref{thm-inner new} we prove that a function $\vp \in \cls(\D^2)$ is inner whenever $\vp = \tau_V$ for some isometric colligation
\[
V=\begin{bmatrix}
a&B_1&B_2\\
C_1&D_{1}&D_{2}\\
C_2&0&D_{3}
\end{bmatrix}
 \in \clb(\mathbb{C} \oplus (\mathcal{H}_1\oplus\mathcal{H}_2)),
\]
with $D_1, D_3 \in C_{0\cdot}$. This is the main content of Section \ref{sec-4}.

\NI The converse of the above fact is not true in general (see Example \ref{example: counter}). However, a weak converse holds for rational inner functions that admit one variable factorization (see Theorem \ref{thm: weak converse}). These are the main content of Section \ref{sec: Counterexamples}.

Now we turn to our second goal of this paper: classification of de Branges-Rovnyak kernels on $\D^n$ and the open unit ball of $\mathbb{C}^n$. Here we explain the idea in the setting of operator-valued Schur functions on $\D$. Suppose $\Theta: \D \raro \clb(\cle, \cle_*)$ is a \textit{Schur function}, that is, $\Theta$ is a $\clb(\cle,\cle_*)$-valued analytic function on $\D$ and $\sup\limits_{z \in \D}\|\Theta(z)\| \leq 1$ (in notation, $\Theta \in \cls(\D, \clb(\cle,\cle_*))$). We call the kernel
\[
K_{\Theta}(z,w):=\frac{I-\Theta(z)\Theta(w)^*}{1-z\overline{w}},\qquad (z,w\in \D).
\]
the \textit{de Branges-Rovnyak kernel} corresponding to $\Theta$. The classical de Branges-Rovnyak theory says that the kernel of a contractively contained shift-invariant (not necessarily closed) subspace of the Hardy space $H^2_{\cle_*}(\D)$ is $K_{\Theta}$ for some $\Theta\in\cls(\D, \clb(\cle,\cle_*))$ and Hilbert space $\cle$.

\NI Section \ref{section: de Branges-Rovnyak} concentrates on the following question: How can we recognize when a kernel admits a de Branges-Rovnyak kernel representation?

\NI The following is our answer to this question (see Theorem \ref{thm: de Branges D}): Let $K \geq 0$ be a $\clb(\cle_*)$-valued kernel (which is not a priori analytic in its first variable). Then $K = K_\Theta$ for some $\Theta\in\cls(\D, \clb(\cle,\cle_*))$ and Hilbert space $\cle$ if and only if
\[
I_{\cle_*}-(1-z\overline{w}) \cdot K\geq 0,
\]
where $\cdot$ denotes the Hadamard product. This also covers a (variation of the) classical result due to de Branges and Rovnyak (see Theorem \ref{cor: de Branges NF} and the discussion preceding it).

\NI In the setting of Schur-Agler functions on $\D^n$ (see more details in Section \ref{section: de Branges-Rovnyak}), in Theorem \ref{thm: de Branges Dn} we prove the following: Let $K: \D^n \times \D^n \raro \clb(\cle_*)$ be a kernel on $\D^n$ (again, $K$ is not a priori analytic in $z_1, \ldots, z_n$). Then there exist a Hilbert space $\cle$ and a $\clb(\cle,\cle_*)$-valued Schur-Agler function $\Theta$ (in notation, $\Theta \in \cls\cla(\D^n,\clb(\cle,\cle_*))$) such that
\[
K = K_{\Theta},
\]
where
\[
K_{\Theta}(\z, \w) := \frac{I - \Theta(\z) \Theta(\w)^*}{\prod\limits_{i=1}^n (1 - z_i \bar{w}_i)} \quad \quad (\z, \w \in \D^n),
\]
if and only if there exist $\clb(\cle_*)$-valued kernels $K_1, \ldots, K_n$ (we call it \textit{Agler kernels} of $\vp$) on $\D^n$ such that
\[
K(\z, \w) = \sum_{i=1}^n \frac{1}{\prod\limits_{j \neq i} (1 - z_j \bar{w}_j)} K_i(\z, \w) \quad \quad (\z, \w \in \D^n),
\]
and
\[
I_{\cle_*} - \Big(\prod\limits_{i=1}^n (1 - z_i \bar{w}_i)\Big) \cdot K \geq 0.
\]
An analogous but somewhat simpler statement also holds in the setting of multipliers of Drury-Arveson space (see Theorem \ref{thm: de Branges Bn}).

The final goal of this paper is to describe those two-variable Schur functions that admit one variable Schur factor. This is the main content of Section \ref{Section: agler kernel}. More specifically (see Theorem \ref{thm: agler kernel and fact}):
Let $\vp \in \cls(\D^2)$ and suppose $\vp(\bm{0})\neq 0$, where $\bm{0} = (0,0)$ (see Remark \ref{rem: phi(0) =0} on the assumption $\vp(\bm{0}) \neq 0$). The following assertions are equivalent:

\NI(1) There exist $\vp_1$ and $\vp_2$ in $\cls(\D)$ such that $\vp(\z)=\vp_1(z_1)\vp_2(z_2)$, $\z \in \D^2$.

\NI (2) There exist Agler kernels $\{K_1, K_2\}$ of $\vp$ such that $K_1$ depends only on $z_1$ and $\bar{w}_1$, and
\[
\overline{\vp(\bm{0})} \, K_2(\cdot,(w_1,0)) = \overline{\vp(w_1,0)} \, K_2(\cdot, \bm{0}) \qquad (w_1 \in \D).
\]
(3)	There exist Agler kernels $\{L_1, L_2\}$ of $\vp$ such that all the functions in $\clh_{L_1}$ depends only on $z_1$, and $\vp(\bm{0}) f(\cdot,0) = \vp(\cdot,0) \, f(\bm{0})$, $f \in \clh_{L_2}$ (here $\clh_K$ refers to the reproducing kernel Hilbert space corresponding to a kernel $K$).

\NI (4) $\vp = \tau_V$ for some co-isometric colligation
\[
V=\begin{bmatrix}
\vp(\bm{0})&B_1&B_2\\C_1&D_1&D_2\\C_2&0&D_4
\end{bmatrix} \in \clb(\mathbb{C}\oplus (\mathcal{H}_1\oplus \mathcal{H}_2)),
\]
with $\vp(\bm{0}) D_2 = C_1B_2$.

We remark that, given the importance of the rich structure, inner functions on the bidisc have been considered in many occasions previously in different contexts (cf. \cite{BL1,SYY}). We also refer to \cite{B et al, GK, HW} and the references therein for the recent and more modern development of Schur functions, Agler kernels, and transfer function realizations.

It is worthwhile to point out that our main motivation of colligation matrices, as in part (4) above and the one following Theorem \ref{thm-Agler D2}, comes from the recent paper \cite{DS}. That said, the present paper is intended as a follow-up, to some extent, to \cite{DS} on one hand, and is also designed to be self-contained on the other.

We end this section with a simple proof of the necessary part of Theorem \ref{thm: classical inner}. The idea of the proof (as also hinted at the final paragraph of this section) is the same as that used in proving analytic representations of commutators of shift operators \cite[Theorem 2]{MSS}. Let $\vp\in \cls(\D)$ be an inner function, and suppose
\[
\vp = \sum_{m= 0}^\infty a_m z^m,
\]
the power series representation of $\vp$ on $\D$. Consider $M_\vp$ as an isometric multiplier on $H^2(\D)$ (that is, $\|M_\vp f\| = \|f\|$ for all $f \in H^2(\D)$), and set $\clq_\vp = H^2(\D) \ominus \vp H^2(\D)$ (that is, $\clq_\vp$ is the orthogonal complement of $\vp H^2(\D)$ in $H^2(\D)$). We clearly have
\[
a_m = P_{\mathbb{C}} M_z^{*m} M_\vp|_{\mathbb{C}} \qquad (m \geq 0),
\]
where $P_{\mathbb{C}}$ denotes the orthogonal projection  onto the space of constant functions in $H^2(\D)$. Note that $M_{\vp}|_{\mathbb{C}} = \vp$.  Since $M_z^* \clq_\vp \subseteq \clq_\vp$ and $M_z^* \vp \in \clq_{\vp}$ (indeed, $\langle M_z^* \vp, \vp f \rangle = \langle M_z^*1, f \rangle = 0$ for all $f \in H^2(\D)$), it follows that
\begin{equation}\label{eqn:phi representation}
\vp(w) =  P_{\mathbb{C}}M_{\vp}|_{\mathbb{C}}+wP_{\mathbb{C}}|_{\clq_{\vp}}(I_{\clq_{\vp}}-wM^*_z|_{\clq_{\vp} })^{-1}M_z^*M_{\vp}|_{\mathbb{C}} \qquad (w \in \D).
\end{equation}
Clearly
\[
V = \begin{bmatrix}\vp(0)&P_{\mathbb{C}}|_{\clq_{\vp}}\\ M_z^*M_{\vp}|_{\mathbb{C}}& M^*_z|_{\clq_{\vp}}\end{bmatrix},
\]
defines a  unitary colligation operator on $\mathbb{C} \oplus \clq_{\vp}$. And, of course, we have $M^*_z|_{\clq_{\vp}} \in C_{0 \cdot}$ and $\vp = \tau_V$.

Note that the representation of $\vp$ in \eqref{eqn:phi representation} reduces to a more compact form as
\[
\vp(w) = P_{\mathbb{C}}(I_{H^2(\D)}-wM^*_z)^{-1}M_{\vp}|_{\mathbb{C}} \qquad (w\in \D).
\]
This formula is comparable to the representation of commutators in the statement of Theorem \cite[Theorem 2]{MSS}. In other words, the above result also follows from Theorem \cite[Theorem 2]{MSS}. However, proof of the sufficient part (as in \cite{F}) of Theorem \ref{thm: classical inner} remains to be involved.

\newsection{Inner Functions and Realizations}\label{sec-4}

Our purpose here is to prove an analogous statement of the sufficient part of Theorem \ref{thm: classical inner}. We will again return to this topic in Section \ref{sec: Counterexamples} with some counterexamples and a weak converse.

It will be convenient, to begin with, some terminology and basic observations. The following construction also could be of some independent interest. We write $\oplus l^2 = l^2 \oplus l^2 \oplus \cdots$, that is
\[
\oplus l^2 = \Big\{ \{a_{ij}\}:= \Big\{ \{ \{a_{0j}\}_{j\geq 0}, \{a_{1j}\}_{j\geq 0}, \{a_{2j}\}_{j\geq 0}, \ldots \} : \sum_{i,j =0}^\infty |a_{ij}|^2 < \infty \Big\}.
\]
One can easily verify that
\[
\tau (\{a_{ij}\})= \sum_{i=0}^{\infty} \sum_{j=0}^{\infty} a_{ij}z_1^iz_2^j,
\]
defines a unitary $\tau: \oplus l^2 \raro H^2(\D^2)$, and $M_{z_1} \tau = \tau S$, where $S$ denotes the \textit{shift} on $\oplus l^2$, that is
\[
S\Big(\{a_{ij}\} \Big) = \Big\{\{0\}, \{a_{0j}\}_{j\geq 0}, \{a_{1j}\}_{j\geq 0}, \ldots \Big\}.
\]
Here $\{0\} \in l^2$ is the zero sequence. Now, let $\vp = \sum\limits_{i=0}^{\infty}(\sum\limits_{j=0}^{\infty} \vp_{ij} z_2^j) z_1^i \in H^\infty (\D^2)$. We define the \textit{block Toeplitz operator} with symbol ${\vp}$ to be the bounded linear operator $T_\vp$ on $\oplus l^2$ defined by
\[
\Big(T_{\vp} \Big(\{a_{ij}\}\Big)\Big)_{ij} = \sum_{k=0}^i \sum_{l=0}^j \vp_{i-k, j-l} a_{kl} \quad \quad (i,j \geq 0),
\]
which in matrix notation becomes
\[
T_{\vp} = \begin{bmatrix}
\Phi_0&0&0&0&\cdots\\
\Phi_1 & \Phi_0&0&0&\cdots\\
\Phi_2& \Phi_1& \Phi_0&0&\cdots\\
\vdots&\vdots&\vdots&\vdots&\ddots
\end{bmatrix},
\]
where
\[
\Phi_k =\begin{bmatrix}
\vp_{k 0} & 0 & 0 & 0 & \cdots
\\
\vp_{k1} & \vp_{k0} & 0 & 0 & \cdots
\\
\vp_{k 2} & \vp_{k 1} & \vp_{k 0} & 0 & \cdots
\\
\vdots&\vdots&\vdots&\vdots&\ddots
\end{bmatrix},
\]
is a Toeplitz operator on $l^2$ for all $k \geq 0$. More specifically, we have
\[
M_\vp \tau = \tau T_{\vp} \qquad (\vp \in H^\infty(\D^2)),
\]
where $M_\vp$ denotes the multiplication operator on $H^2(\D^2)$ with analytic symbol $\vp$, that is, $M_\vp f = \vp f$ for all $f \in H^2(\D^2)$. Indeed, if $\vp = \sum\limits_{i=0}^{\infty}(\sum\limits_{j=0}^{\infty} \vp_{ij} z_2^j) z_1^i$ and $\{a_{ij}\} \in \oplus l^2$, then
\[
\begin{split}
M_{\vp} \tau (\{a_{ij}\}) &= \left( \sum_{i,j =0}^{\infty} \vp_{ij} z_1^iz_2^j \right)
\left( \sum_{k, l=0}^{\infty} a_{kl} z_1^k z_2^l \right)
\\
& = \sum_{j, l=0}^{\infty} \; \sum_{i, k=0}^{\infty} \vp_{ij}  a_{kl} z_1^{i+k}z_2^{l+j}
\\
& = \sum_{j, l=0}^{\infty} \left\{ \sum_{i=0}^{\infty} \left(\sum_{k=0}^i \vp_{i-k, j} a_{kl}\right) z_1^i \right\} z_2^{l+j}
\\
& = \sum_{i=0}^{\infty} \sum_{k=0}^i \left\{ \sum_{j, l=0}^{\infty} \vp_{i-k, j} a_{kl} z_2^{l+j}  \right\} z_1^i
\\
& = \sum_{i=0}^{\infty} \sum_{k=0}^i \left\{\sum_{j=0}^{\infty} \left(\sum_{l=0}^j \vp_{i-k, j-l} a_{kl}\right) z_2^j \right\} z_1^i
\\
& = \sum_{i, j=0}^{\infty} \left( \sum_{k, l=0}^{i,j} \vp_{i-k, j-l} a_{kl} \right)  z_1^iz_2^j,
\end{split}
\]
and hence $M_{\vp} \tau (\{a_{ij}\}) = \tau T_\vp (\{a_{ij}\})$. In particular, we have
\[
T_{z_2} = \begin{bmatrix}
S_{l^2} &0&0&0&\cdots\\
0 & S_{l^2} &0&0&\cdots\\
0 & 0 & S_{l^2} &0&\cdots\\
\vdots&\vdots&\vdots&\vdots&\ddots
\end{bmatrix},
\]
where $S_{l^2}$ denotes the shift on $l^2$, that is, $S_{l^2}(\{a_0, a_1, \ldots \}) = \{0, a_0, a_1, \ldots\}$ for all $\{a_m\}_{m\geq 0} \in l^2$. Continuing with the above notation, we set
\begin{equation}\label{eqn: Y_0 and Y_j}
Y_0=\begin{bmatrix}
\Phi_0 \\ \Phi_1 \\ \Phi_2 \\ \vdots
\end{bmatrix}, \quad \mbox{and} \quad Y_j = S^jY_0,
\end{equation}
for all $j \geq 1$. Then $T_\vp = \begin{bmatrix} Y_0 & Y_1 & Y_2 & \ldots \end{bmatrix}$. Since $T_\vp^* T_\vp = (Y_i^* Y_j)$, it follows that $M_{\vp}$ on $H^2(\D^2)$ is an isometry if and only if $T_{\vp}$ on $\oplus l^2$ is an isometry, which is also equivalent to
\begin{equation}\label{eqn: YiYi*=delta I}
Y_i^* Y_j = \delta_{ij} I_{l^2}.
\end{equation}
We are now ready to present the main theorem of this section.

\begin{thm}\label{thm-inner new}
Let $\vp\in \cls(\D^2)$. If $\vp = \tau_V$ for some isometric colligation
\[
V=\begin{bmatrix}
a&B_1&B_2\\
C_1&D_{1}&D_{2}\\
C_2&0&D_{3}
\end{bmatrix}
:\mathbb{C} \oplus (\mathcal{H}_1\oplus\mathcal{H}_2) \rightarrow \mathbb{C} \oplus (\mathcal{H}_1\oplus\mathcal{H}_2),
\]
with $D_1, D_3 \in C_{0\cdot}$, then $\vp$ is an inner function.
\end{thm}
\begin{proof} Since $\vp=\tau_V$, and
\[
\tau_V(\z) = a + \begin{bmatrix} B_1 & B_2 \end{bmatrix} \Big(I_{\clh_1 \oplus \clh_2} - E_{\clh_1 \oplus \clh_2}(\z) \begin{bmatrix} D_{1}&D_{2}\\ 0&D_{3}\end{bmatrix}\Big)^{-1} E_{\clh_1 \oplus \clh_2}(\z) \begin{bmatrix}C_1\\C_2\end{bmatrix},
\]
we have
\[
\vp(\z) = a + \sum\limits_{i=1}^{\infty}B_1D_1^{i-1} C_1 z_1^{i} + \sum\limits_{j= 1}^{\infty}B_2 D_3^{j-1}C_2z_2^{j} + \sum\limits_{i= 1}^{\infty}\sum\limits_{j= 1}^{\infty}B_1 D_{1}^{i-1} D_{2} D_{3}^{j-1} C_2 z_1^{i} z_2^{j},
\]
for all $\z \in \D^2$. Under the same notations preceding the statement, we set
\[
\Phi_0=\begin{bmatrix}
a&0&0&0&\cdots
\\
B_2 C_2 & a & 0 & 0 & \cdots
\\
B_2 D_{3} C_2 & B_2 C_2&a&0&\cdots
\\
B_2 D_{3}^2C_2& B_2 D_{3}C_2 & B_2 C_2&a&\cdots
\\
\vdots&\vdots&\vdots&\vdots&\ddots
\end{bmatrix},
\]
and
\[
\Phi_j=\begin{bmatrix}
B_1 D_{1}^{j-1} C_1 & 0& 0&0&\cdots
\\
B_1 D_{1}^{j-1}D_{2}C_2& B_1 D_{1}^{j-1} C_1&0&0&\cdots
\\
B_1 D_{1}^{j-1} D_{2} D_{3} C_2& B_1 D_{1}^{j-1} D_{2}C_2&B_1D_{1}^{j-1}C_1&0&\cdots
\\
B_1 D_{1}^{j-1} D_{2} D_{3}^2C_2 & B_1 D_{1}^{j-1} D_{2}D_{3}C_2 & B_1 D_{1}^{j-1} D_{2} C_2 & B_1 D_{1}^{j-1} C_1&\cdots
\\
\vdots&\vdots&\vdots&\vdots&\ddots
\end{bmatrix},
\]
for $j\geq 1$. We first claim that $Y_0=\begin{bmatrix}
\Phi_0\\
\Phi_1\\
\Phi_2\\
\vdots
\end{bmatrix}$ is an isometry. In fact, since $Y_0^*Y_0 = \sum\limits_{m=0}^\infty \Phi_m^* \Phi_m$, there exists a sequence of scalars $\{y_m\}_{m\geq 0}$ such that
\[
Y_0^*Y_0 = \begin{bmatrix}
y_0&y_1&y_2&\cdots\\
\overline{y_1}&y_0&y_1&\cdots\\
\overline{y_2}&\overline{y_1}&y_0&\cdots\\
\vdots&\vdots&\vdots&\ddots
\end{bmatrix}.
\]	
We need to show that $y_0 =1$ and $y_k = 0$ for all $k \geq 1$. Note that
\[
\begin{split}
y_0 & = |a|^2 + C_1^*\left(\sum\limits_{j=0}^\infty D_1^{*j} B_1^* B_1 D_1^j \right) C_1
\\
& \quad + C_2^*\left[\sum\limits_{k = 0}^\infty D_3^{*k} \bigg\{ B_2^*B_2 + D^*_2 \left(\sum\limits_{l\geq 0}D_1^{*l} B_1^* B_1 D_1^l \right) D_2 \bigg\} D_3^k \right]C_2.
\end{split}
\]
Since $V^* V = I$, it follows that
\begin{equation}\label{eq-VstarV}
\begin{bmatrix} |a|^2 + C_1^* C_1 + C_2^* C_2 & \bar{a} B_1 + C_1^* D_1 & \bar{a} B_2 + C_1^* D_2 + C_2^* D_3
\\
a B_1^* + D_1^* C_1 & B_1^* B_1 + D_1^* D_1 & B_1^* B_2 + D_{1}^* D_2
\\
a B_2^* + D_2^* C_1 + D_3^* C_2 & B_2^* B_1 + D_2^*D_1 & B_2^* B_2 + D_2^* D_2 + D_3^* D_3
\end{bmatrix} = I.
\end{equation}
In particular
\[
\begin{split}
I & = B_1^* B_1 + D_1^* D_1
\\
& = B_1^* B_1 + D_1^* (B_1^* B_1 + D_1^* D_1) D_1
\\
& = B_1^* B_1 + D_1^* B_1^* B_1 D_1 + D_1^{*2} D_1^2,
\end{split}
\]
and hence $I = \sum\limits_{j = 0}^m D_{1}^{*j}(B_1^*B_1)D_{1}^j + D_{1}^{*(m+1)}D_{1}^{(m+1)}$ for all $m \geq 1$. Using the fact that $D_1 \in C_{0\cdot}$, we have
\begin{equation}\label{eqn: infinte sum 1 =1}
\sum\limits_{j = 0}^\infty D_1^{*j}(B_1^*B_1) D_1^j = I,
\end{equation}
in the strong operator topology. Similarly, $B_2^* B_2 + D_2^* D_2 + D_3^* D_3 = I$ and $D_3 \in C_{0\cdot}$ implies that
\begin{equation}\label{eqn: infinte sum 2 =1}
\sum\limits_{j=0}^\infty D_{3}^{*j}(B_2^*B_2+D_2^*D_2)D_{3}^j = I,
\end{equation}
in the strong operator topology. This with the condition $|a|^2 + C_1^*C_1 + C_2^*C_2 = 1$ in \eqref{eq-VstarV} implies that
\[
y_0 = |a|^2 + C_1^*C_1 + C_2^*C_2 = 1.
\]
Next we consider
\[
\begin{split}
y_1 & = aC_2^* B_2^* + C_2^*D_{2}^*\left(\sum\limits_{j= 0}^\infty D_1^{*j} B_1^*B_1 D_1^j\right)C_1
\\&\quad +C_2^*D_3^* \left[\sum\limits_{k = 0}^\infty D_3^{*k}\bigg\{ B_2^*B_2 + D^*_2 \left(\sum\limits_{l=0}^\infty D_1^{*l}B_1^*B_1 D_1^l\right)D_2 \bigg\}D_3^k \right]C_2 \,.
\end{split}
\]
Thus by \eqref{eqn: infinte sum 1 =1} and \eqref{eqn: infinte sum 2 =1}, it follows that
\[
y_1 = a C_2^* B_2^* + C_2^*D_2^*C_1+C_2^*D_3^* C_2 = C_2^*(aB_2^*+D_2^*C_1 + D_3^* C_2) =0,
\]
as $aB_2^*+D_2^*C_1 + D_3^* C_2 = 0$ follows from \eqref{eq-VstarV}. Similarly
\[
y_j = C_2^*D_3^{*j} (aB_2^* + D_2^* C_1 + D_3^* C_2)=0,
\]
for all $j\geq2$. This proves that $Y_0$ is an isometry.

\NI Since the shift $S$ on $\oplus l^2$ is an isometry, $Y_j :=S^jY_0$, $j\geq1$, is also an isometry (see the construction in \eqref{eqn: Y_0 and Y_j}). Our final goal is to prove that $T_\vp := \begin{bmatrix} Y_0 & Y_1 & Y_2 & \ldots \end{bmatrix}$ is an isometry, or equivalently, by virtue of \eqref{eqn: YiYi*=delta I} and $Y_m^* Y_m = I$ for all $m \geq 0$,
\[
Y_p^* Y_q = 0 \quad \quad (p>q \geq 0).
\]
Since $Y_p^* Y_q = Y_0^* S^{*p} S^q Y_0 = Y_0^* S^{*(p-q)} Y_0$ for all $p>q \geq 0$, it actually suffices to check that
\[
Y_0^* S^{*(j+1)} Y_0 = 0 \qquad\qquad (j \geq 0).
\]
So we fix $j \geq 0$ and observe
\[
S^jY_0=\left[
\begin{array}{ccccc}
\onegroup& \Phi_0 & \Phi_1 & \Phi_2 & \ldots
\end{array}
\right]^t.
\]
Hence
\[
Y_0^* S^{*(j+1)} Y_0 = \Phi_0^* \Phi_{j+1}+ \Phi_1^* \Phi_{j+2} + \cdots.
\]
Therefore there exists a sequence $\{c_m\}_{m \in \mathbb{Z}}$ such that
\[
Y_0^* S^{*(j+1)} Y_0 = \begin{bmatrix}
c_0&c_1&c_2&\cdots
\\
c_{-1}&c_0&c_1&\cdots
\\
c_{-2}&c_{-1}&c_{0}&\cdots
\\
\vdots&\vdots&\vdots&\ddots
\end{bmatrix}.
\]
It is suffices to prove that $c_k = 0$ for all $k \in \mathbb{Z}$. A simple calculation shows that
\[
c_0 = (\bar{a}B_1+C_1^*D_1)D_1^{j+1}C_1 + C_2^* \left[\sum\limits_{m=0}^\infty D_3^{*m} \bigg\{B_2^*B_1  +D_2^*D_1 \bigg\} D_1^{j+1} D_2 D_3^m \right]C_2.
\]
By \eqref{eq-VstarV}, $\bar{a}B_1+C_1^*D_1 = 0$ and $B_2^*B_1  +D_2^*D_1 = 0$, and hence $c_0 = 0$. Now let $k >0$. Then
\[
c_k = C_2^*D_3^{*(k-1)} (B_2^*B_1 + D_2^* D_1) D_1^{j+1} C_1 + C_2^* D_3^{*k} \left[ \sum\limits_{m= 0}^\infty D_3^{*m} \bigg\{B_2^*B_1 +D_2^*D_1 \bigg\}D_1^{j+1}D_2 D_3^m \right]C_2,
\]
and hence $c_k = 0$. Finally, since
\[
c_{-k} = (\overline{a}B_1 + C_1^*D_1) D_1^{j+1} D_2 D_3^{k-1} C_2 + C_2^*\left[ \sum\limits_{m = 0}^\infty D_3^{*m} \bigg\{B_2^*B_1 + D_2^*D_{1,1} \bigg\}D_1^{j+1}D_2 D_3^m \right] D_3^k C_2,
\]
it again follows that $c_{-k} = 0$. This implies that $T_{\vp}$ or,
equivalently, $M_{\vp}$ is an isometry, and completes the proof.
\end{proof}

\begin{rem}\label{rem: D C dot zero}
Let $V \in \clb(\mathbb{C} \oplus \clh)$ be an isometric colligation, and let $\clh = \clh_1 \oplus \clh_2$ for some Hilbert spaces $\clh_1$ and $\clh_2$. Suppose $D:= P_{\clh} V|_{\clh}$ and suppose that $D \clh_1 \subseteq \clh_1$. Set
\[
D=\begin{bmatrix} D_1&D_2\\0&D_3\end{bmatrix}  \in \clb(\clh_1 \oplus \clh_2).
\]
It is easy to see that if $D \in C_{0 \cdot}$, then $D_1$ and $D_3$ are also in $C_{0 \cdot}$. Consequently, Theorem \ref{thm-inner new} also holds for those Schur functions $\vp$ such that $\vp = \tau_V$ with $V$ as above. Of course, if $D_1$ and $D_3$ are in $C_{0 \cdot}$, then $D$ is not necessarily in $C_{0 \cdot}$.
\end{rem}

\newsection{de Branges-Rovnyak kernels}\label{section: de Branges-Rovnyak}

The goal of this section is to study de Branges-Rovnyak kernels on $\D^n$ and the open unit ball of $\mathbb{C}^n$, $n \geq 1$. Specifically, we seek characterizations of analytic kernels that admit certain factorizations involving Schur(-Agler) functions. Our investigation is partly motivated by a classical result of de Branges and Rovnyak (see the paragraph preceding Theorem \ref{cor: de Branges NF} for more details).

We start with the unit disc case. Let $\cls(\D, \clb(\cle,\cle_*))$ denote the set of all $\clb(\cle,\cle_*)$-valued analytic functions $\Theta$  on $\D$ such that $\sup_{z \in \D}\|\Theta(z)\| \leq 1$. Such functions are called \textit{operator-valued Schur functions}.

A kernel $K : \D \times \D \raro \clb(\cle)$ is a \textit{de Branges-Rovnyak kernel} if there exists $\Theta\in\cls(\D, \clb(\cle,\cle_*))$ such that
\[
K(z,w) = K_{\Theta}(z,w):=\frac{I-\Theta(z)\Theta(w)^*}{1-z\bar{w}} \qquad (z,w\in \D).
\]
Note that if $\Theta \in \cls(\D, \clb(\cle,\cle_*))$, then $M_\Theta : H^2_{\cle}(\D) \raro H^2_{\cle_*}(\D)$ is a contraction. This is clearly equivalent to the condition that $K_{\Theta} \geq 0$.

In the following, we characterize de Branges-Rovnyak kernels defined on the disc $\D$. The proof uses the commonly used ``lurking-isometry'' techniques. Therefore, our proof is fairly standard and, perhaps, it can also be achieved from existing results of Schur(-Agler) functions \cite{AM}. Note also that the theorem below does not assume a priori that $K$ is analytic in its first variable.

\begin{thm}\label{thm: de Branges D}
Let $K: \D \times \D \raro \clb(\cle_*)$ be a kernel on $\D$. Then $K=K_{\Theta}$  for some $\Theta \in \cls(\D,\clb(\cle,\cle_*))$ and Hilbert space $\cle_*$ if and only if
\[
I_{\cle_*} - (1-z\bar{w}) \cdot K \geq 0.
\]	
\end{thm}
\begin{proof}
If $K=K_{\Theta}$, then
\[
I_{\cle_*}-(1-z\bar{w})K(z,w) = \Theta(z)\Theta(w)^* \geq 0 \qquad (z, w \in \D).
\]
Conversely, if $I_{\cle_*}-(1-z\overline{w}) \cdot K\geq 0$, then there exist a Hilbert space $\clf$ and a function (a priori not necessarily analytic) $F: \D \rightarrow \clb(\clf,\cle_*)$ such that
\[
I_{\cle}-(1-z\bar{w}) K(z, w) = F(z)F(w)^* \qquad (z, w \in \D).
\]
Clearly, $F$ is a contractive function on $\D$. Again, since $K\geq 0$, there exist a Hilbert space $\clg$ and a function $G : \D \rightarrow \clb(\clg,\cle_*)$ such that $K(z,w)=G(z)G(w)^*$, $z, w \in \D$. Then
\[
I_{\cle_*} - G(z)G(w)^* + z\bar{w}G(z)G(w)^* = F(z)F(w)^*,
\]
and hence
\[
I_{\cle_*}+z \bar{w}G(z)G(w)^* = G(z)G(w)^* + F(z)F(w)^*,
\]
for all $z, w \in \D$. Therefore
\[
V:\begin{bmatrix}
I_{\cle_*}\\ \bar{w}G(w)^*\end{bmatrix}\eta \mapsto \begin{bmatrix}F(w)^*\\G(w)^*
\end{bmatrix}\eta \qquad \qquad (w \in \D, \eta\in \cle_*),
\]
defines an isometry from a subspace of $\cle_*\oplus\clg$ to $\clf\oplus\clg$. Then, adding an infinite-dimensional summand to $\clg$ if necessary, $V$ can then be extended to an isometry, denoted by $V$ again, from $\cle_*\oplus\clg$ to $\clf\oplus\clg$. Set
\[
V = \begin{bmatrix}A&B\\C&D \end{bmatrix}: \cle_* \oplus \clg \rightarrow \clf \oplus \clg.
\]
Then
\[
\begin{bmatrix}A&B\\C&D \end{bmatrix} \begin{bmatrix}
\eta \\ \bar{w}G(w)^* \eta \end{bmatrix} = \begin{bmatrix}F(w)^*\eta \\ G(w)^* \eta
\end{bmatrix},
\]
for all $\eta\in \cle$ and $w \in \D$, which implies that
\[
A+\bar{w}BG(w)^*=F(w)^* \;\mbox{and} \; C+\bar{w}DG(w)^*=G(w)^*,
\]
for all $w \in \D$. The latter equality implies that $G(w)^*=(I-\overline{w}D)^{-1}C$, and hence, the first equality yields
\[
F(w)^*= A+\overline{w}B(I-\overline{w}D)^{-1}C,
\]
for all $w \in \D$. Hence
\[
F(z)=A^*+zC^*(I-zD^*)^{-1}B^*\qquad (z\in \D),
\]
that is, $F = \tau_{V^*}$ is analytic on $\D$ and bounded by $1$, where
\[
V^* = \begin{bmatrix}A^*& C^* \\ B^* & D^*\end{bmatrix},
\]
is a co-isometric colligation. Consequently, $\Theta:= F \in \cls(\D, \clb(\clf,\cle_*))$, and hence
\[
I_{\cle_*}-(1-z \bar{w}) K(z,w) = \Theta(z)\Theta(w)^*,
\]
that is, $K(z,w)=\dfrac{I_{\cle_*}-\Theta(z)\Theta(w)^*}{1-z\bar{w}}$ for all $z, w \in \D$. This completes the proof.
\end{proof}

We denote by $\mathbb{S}_n$ the Szeg\"{o} kernel on $\D^n$, that is
\[
\mathbb{S}_n(\z, \w) = \prod_{i=1}^{n}\frac{1}{1 - z_i \bar{w}_i} \qquad (\z, \w \in \D^n).
\]
Also we denote $\mathbb{S}_1$ simply by $\mathbb{S}$. The following is a variation of a result due to de Branges and Rovnyak \cite{dBR1, dBR2}. Also, we refer the reader to the classic Sz.-Nagy and Foias \cite[Section 8, page 231]{NF} for detailed proof and some historical notes. The proof below follows the proof of the previous theorem. Again, a priori we do not assume (in contrast to Sz.-Nagy and Foias) that $K$ is analytic in its first variable.

\begin{thm}\label{cor: de Branges NF}
Let $K:\D\times\D\rightarrow \clb(\cle_*)$ be a kernel. Then
\[
0\leq K \leq \mathbb{S} \text{ and }  \mathbb{S}^{-1} \cdot K\geq 0,
\]
if and only if there exist a Hilbert space $\cle$ and an operator-valued Schur function $\Theta\in \cls(\D, \clb(\cle, \cle_*))$ such that
\[
K(z,w)=\frac{\Theta(z)\Theta(w)^*}{1-z\bar{w}} \qquad (z, w \in \D).
\]
 \end{thm}
 \begin{proof}
Suppose $0\leq K \leq \mathbb{S}$ and $\mathbb{S}^{-1} \cdot K\geq 0$. Now $0\leq K \leq \mathbb{S}$ implies that
\[
\frac{1}{1 - z \bar{w}}I - K(z,w)\geq 0.
\]
As in the proof of the previous theorem, there exist a Hilbert space $\clf$ and a function $G: \D \raro \clb(\clf, \cle_*)$ such that
\[
I-(1-z \bar{w})K(z,w) = (1-z\bar{w}) G(z)G(w)^*.
\]
Again, since $\mathbb{S}^{-1} \cdot K\geq 0$, there exist a Hilbert space $\clg$ and a function $F : \D \raro \clb(\clg, \cle_*)$ such that
\[
I-F(z)F(w)^* = (1-z\bar{w})G(z)G(w)^*
\]
The remaining argument is similar to that of the proof of the previous theorem.
\end{proof}

Now we recall the definition of Schur-Agler functions. Let $\cle$ and $\cle_*$ be Hilbert spaces. The \textit{Schur-Agler class} $\mathcal{SA}(\D^n,\clb(\cle,\cle_*))$ \cite{JA1} consists of $\clb(\cle,\cle_*)$-valued analytic functions $\vp$ on $\D^n$ such that $\vp$ satisfies the $n$-variables von Neumann inequality
\[
\|\vp(T_1, \ldots, T_n)\|_{\clb(\clh)} \leq 1,
\]
for any $n$-tuples of commuting strict contractions on a Hilbert space $\clh$. Here
\[
\vp(T_1, \ldots, T_n) = \sum_{\bm{k} \in \mathbb{Z}_+^n} \vp_{\bm{k}} \otimes T^{\bm{k}},
\]
where $\vp = \sum\limits_{\bm{k} \in \mathbb{Z}_+^n} \vp_{\bm{k}}\z^{\bm{k}}$, $\vp_{\bm{k}} \in \clb(\cle, \cle_*)$, and $T^{\bm{k}} = T_1^{k_1} \cdots T_n^{k_n}$ for all $\bm{k} = (k_1, \ldots, k_n) \in \mathbb{Z}_+^n$. The elements of $\mathcal{SA}(\D^n,\clb(\cle,\cle_*))$  are called \textit{Schur-Agler functions}. The following result is due to Agler \cite{JA1} (also see Theorem \ref{thm-Agler D2}):

\NI Given a function $\Theta: \D^n \raro \clb(\cle,\cle_*)$, the following are equivalent:

\NI (i) $\Theta \in \mathcal{SA}(\D^n, \clb(\cle,\cle_*))$.

\NI (ii) There exist $\clb(\cle_*)$-valued kernels $K_1, \ldots, K_n$ (known as \textit{Agler kernels}) on $\D^n$ such that
\[
I_{\cle_*} - \Theta(\z) \Theta(\w)^* = \sum_{i=1}^n (1 - z_i \bar{w}_i) K_i(\z, \w), \qquad (\z, \w \in \D^n).
\]

We now turn to de Branges-Rovnyak kernels on $\D^n$. Suppose $\Theta \in \mathcal{SA}(\D^n, \clb(\cle,\cle_*))$. Since $M_{\Theta}$ is a contraction from $H^2_{\cle}(\D^n)$ into $H^2_{\cle_*}(\D^n)$, it is easy to check (as also pointed out earlier) that $K_{\Theta}\geq 0$, where
\[
K_{\Theta}(\z, \w) = \mathbb{S}_n(\z, \w)^{-1} (I-\Theta(\z) \Theta(\w)^*) \quad \quad (\z, \w \in \D^n).
\]
Here we say that $K_{\Theta}$ is a ($\clb(\cle_*)$-valued) de Branges-Rovnyak kernel on $\D^n$. In the following, we do not assume a priori that $K$ is analytic in $z_1, \ldots, z_n$.

\begin{thm}\label{thm: de Branges Dn}
Let $K: \D^n \times \D^n \raro \clb(\cle_*)$ be a kernel on $\D^n$. Then $K=K_{\Theta}$ for some Schur-Agler function $\Theta \in \cls\cla(\D^n,\clb(\cle,\cle_*))$ and a  Hilbert space $\cle$ if and only if there exist $\clb(\cle_*)$-valued kernels $K_1, \ldots, K_n$ on $\D^n$ such that
\[
K(\z, \w) = \sum_{i=1}^n \frac{1}{\prod\limits_{j \neq i} (1 - z_j \bar{w}_j)} K_i(\z, \w),
\]
for all $\z, \w \in \D^n$, and $I_{\cle_*} - \mathbb{S}_n^{-1} \cdot K \geq 0$.	
\end{thm}
\begin{proof}
The “only if” part of this statement is easy, and the proof of the “if” part is similar to the proof of Theorem \ref{thm: de Branges D}. We give only a sketch: Suppose $K_1, \ldots, K_n$ are $\clb(\cle_*)$-valued kernels on $\D^n$, $\z, \w \in \D^n$, and suppose
\[
K(\z, \w) = \sum_{i=1}^n \frac{1}{\prod\limits_{j \neq i} (1 - z_j \bar{w}_j)} K_i(\z, \w).
\]
Then
\[
\mathbb{S}_n^{-1}(\z, \w) K(\z, \w) = \sum\limits_{i=1}^n(1-z_i\bar{w}_i) K_i(\z, \w).
\]
Since $I_{\cle_*} - \mathbb{S}_n^{-1} \cdot K\geq 0$, there exist a Hilbert space $\clg$ and a function $G: \D^n \raro \clb(\clg, \cle_*)$ such that
\[
I_{\cle} - \mathbb{S}_n^{-1}(\z, \w) K(\z, \w) = G(\z) G(\w)^*.
\]
Again, since $K_i\geq 0$, there exist Hilbert spaces $\clf_1, \ldots, \clf_n$, and functions $F_i: \D^n \raro \clb(\clf_i, \cle_*)$, $i=1, \ldots, n$, such that $K_i(\z, \w) = F_i(\z) F_i(\w)^*$ for all $i=1, \ldots, n$. Hence
\[
\mathbb{S}_n^{-1}(\z, \w) K(\z, \w) = \sum\limits_{i}^n(1-z_i\bar{w}_i) F_i(\z) F_i(\w)^*,
\]
which implies
\[
I_{\cle_*} + \sum\limits_{i=1}^n z_i \bar{w}_i F_i(\z) F_i(\w)^* = G(\z) G(\w)^* + \sum\limits_{i=1}^n F_i(\z) F_i(\w)^*,
\]
for all $\z, \w \in \D^n$. Now one can proceed with the lurking-isometry method, as in the proof of Theorem \ref{thm: de Branges D}, to complete the proof of the theorem.
\end{proof}

An analogous statement also holds in the case of multipliers of \textit{Drury-Arveson space} $H^2_n$. We recall that $H^2_n$ is the reproducing kernel Hilbert space corresponding to the kernel
\[
S(\z, \w) := \frac{1}{1-\langle \z, \w\rangle} \qquad (\z, \w \in \mathbb{B}^n),
\]
where $\mathbb{B}^n = \{\z \in \mathbb{C}^n: \sum\limits_{i=1}^{n}|z_i|^2 < 1\}$ is the open unit ball of $\mathbb{C}^n$, and $\langle \z, \w\rangle = \sum\limits_{i=1}^n z_i \bar{w}_i$. Given Hilbert spaces $\cle$ and $\cle_*$, the $\clb(\cle, \cle_*)$-valued Drury-Arveson multiplier space is defined by
\[
\clm_n(\cle,\cle_*)=\{\Theta: \mathbb{B}^n\rightarrow \clb(\cle,\cle_*): \Theta(H^2_n \otimes \cle) \subseteq  H^2_n\otimes\cle_*\}.
\]
Here $\clm_n(\cle,\cle_*)$ is a Banach space equipped with the norm $\|\Theta\|_{\clm_n(\cle,\cle_*)} = \|M_{\Theta}\|$ (the operator norm of $M_\Theta$). In this setting, the de Branges-Rovnyak kernel $K_{\Theta}$ corresponding to $\Theta \in \clm_d(\cle,\cle_*)$ is defined by
\[
K_{\Theta}(\z, \w) = \frac{I-\Theta(\z) \Theta(\w)^*}{1 - \langle \z, \w \rangle} \qquad (\z, \w \in \mathbb{B}^n).
\]
The proof of the following theorem is completely analogous to the proof of Theorems \ref{thm: de Branges D} and \ref{thm: de Branges Dn}. We leave details to the reader.

\begin{thm}\label{thm: de Branges Bn}
Let $\cle_*$ be a Hilbert space and $K:\mathbb{B}^n\times \mathbb{B}^n\rightarrow \clb(\cle)$ be a kernel. Then $K = K_{\Theta}$ for some $\Theta \in \clm_n(\cle,\cle_*)$ and Hilbert space $\cle$ if and only if
\[
I_{\cle_*} - (1 - \langle \z, \w \rangle) \cdot K(\z, \w) \geq 0.
\]
\end{thm}

In the above theorem, we do not assume a priori that $K$ is analytic in $z_1, \ldots, z_n$.

\section{Agler Kernels and Factorizations}\label{Section: agler kernel}

In this section we investigate factorizations of two-variable Schur functions in terms of Agler kernels. We shall be particularly interested in the case of one variable factors and Agler kernels of functions in $\cls(\D^2)$.

Here and in what follows, $\clh_K$ will denote the reproducing kernel Hilbert space corresponding to the kernel $K$. Moreover, if $K: \D^2 \times \D^2 \raro \mathbb{C}$, then $K(\cdot, \w) \in \clh_K$ will denote the kernel function at $\w \in \D^2$, that is
\[
\Big(K(\cdot, \w)\Big)(\z) = K(\z, \w) \qquad (\z \in \D^2),
\]
and
\[
f(\w) = \langle f, K(\cdot, \w) \rangle_{\clh_K},
\]
for all $f \in \clh_K$ and $\w \in \D^2$. For notational convenience we write $\bm{0}= (0,0)$.

We are now ready for the main result of this section (see Remark \ref{rem: phi(0) =0} on the assumption $\vp(\bm{0}) \neq 0$):

\begin{thm}\label{thm: agler kernel and fact}
Let $\vp \in \cls(\D^2)$ and suppose $\vp(\bm{0})\neq 0$. The following assertions are equivalent:

\NI(1) There exist $\vp_1$ and $\vp_2$ in $\cls(\D)$ such that
\[
\vp(\z)=\vp_1(z_1)\vp_2(z_2) \qquad (\z \in \D^2).
\]
(2) There exist Agler kernels $\{K_1, K_2\}$ of $\vp$ such that $K_1$ depends only on $z_1$ and $\bar{w}_1$, and
\[
\overline{\vp(\bm{0})} \, K_2(\cdot,(w_1,0)) = \overline{\vp(w_1,0)} \, K_2(\cdot, \bm{0}) \qquad (w_1 \in \D).
\]
(3)	There exist Agler kernels $\{L_1, L_2\}$ of $\vp$ such that all the functions in $\clh_{L_1}$ depends only on $z_1$, and
\[
\vp(\bm{0}) f(\cdot,0) = \vp(\cdot,0) \, f(\bm{0})\qquad (f \in \clh_{L_2}).
\]	
(4) $\vp = \tau_V$ for some co-isometric colligation
\[
V=\begin{bmatrix}
\vp(\bm{0})&B_1&B_2\\C_1&D_1&D_2\\C_2&0&D_4
\end{bmatrix} \in \clb(\mathbb{C}\oplus (\mathcal{H}_1\oplus \mathcal{H}_2)),
\]
with $\vp(\bm{0}) D_2 = C_1B_2$.
\end{thm}
\begin{proof}
Suppose first that $\vp(\z)=\vp_1(z_1)\vp_2(z_2)$, $\z \in \D^2$, for some $\vp_1$ and $\vp_2$ in $\cls(\D)$. Then
\[
1- \vp(\z)\overline{\vp(\w)} = 1-\vp_1(z_1) \overline{\vp_1(w_1)} + \vp_1(z_1) (1- \vp_2(z_2) \overline{\vp_2(w_2)})\, \overline{\vp_1(w_1)},
\]
and hence
\[
1- \vp(\z)\overline{\vp(\w)} = (1- z_1 \bar{w}_1) K_1(\z,\w) + (1 - z_2 \bar{w}_2) K_2(\z, \w),
\]
where
\[
K_1(\z, \w) = \frac{1-\vp_1(z_1)\overline{\vp_1(w_1)}}{1-z_1\bar{w}_1} \quad
\text{and} \quad
K_2(\z, \w)= \frac{\vp_1(z_1) (1-\vp_2(z_2)\overline{\vp_2(w_2)})\, \overline{\vp_1(w_1)}}{1- z_2 \bar{w}_2},
\]
and $\z, \w \in \D^2$. Then $\{K_1, K_2\}$ are Agler kernels of $\vp$ and satisfies the conditions of (2). This proves (1)$\Rightarrow$(2).

\NI (2)$\Rightarrow$ (3): Set $L_i= K_i$, $i=1,2$, and suppose $f \in \clh_{K_1}$. Since
\[
\w \mapsto f(w) = \langle f, L_1(\cdot, \w) \rangle = \langle f, K_1(\cdot, \w) \rangle,
\]
and $K_1$ depends only on $z_1$ and $\bar{w}_1$, it follows that all the functions in $\clh_{L_1}$ depends only on $z_1$. On the other hand, if $f \in \clh_{L_2}$, then
\[
\begin{split}
\vp(\bm{0}) f(w_1,0) & = \langle f, \overline{\vp(\bm{0})} L_2(\cdot, (w_1,0)) \rangle_{\clh_{L_2}}
\\
& = \langle f, \overline{\vp(\bm{0})} K_2(\cdot, (w_1,0)) \rangle_{\clh_{K_2}}
\\
& = \langle f, \overline{\vp(w_1,0)} K_2(\cdot, (\bm{0})) \rangle_{\clh_{K_2}}
\\
& = \vp(w_1,0) \langle f, K_2(\cdot, (\bm{0})) \rangle_{\clh_{K_2}},
\end{split}
\]
and hence $\vp(\bm{0})f(w_1,0) = \vp(w_1,0) f(\bm{0})$ for all $w_1 \in \D$.

\NI (3)$\Rightarrow$ (2): This is just the reverse of the argument in the above proof.

\NI (2)$\Rightarrow$ (4): Suppose $\{K_1, K_2\}$ are Agler kernels of $\vp$, and suppose that $K_1$ depends only on $z_1$ and $\bar{w}_1$, and
\begin{equation}\label{eqn: agler condition 4}
\overline{\vp(\bm{0})} \, K_2(\cdot,(w_1,0)) = \overline{\vp(w_1,0)} \, K_2(\cdot, \bm{0}) \qquad (w_1 \in \D).
\end{equation}
Now
\[
1-\vp(\z) \overline{\vp(\w)} = (1-z_1 \bar{w}_1) \langle K_1(\cdot, \w), K_1(\cdot, \z) \rangle_{\clh_{K_1}} + (1-z_2 \bar{w}_2) \langle K_2(\cdot, \w), K_2(\cdot, \z) \rangle_{\clh_{K_2}},
\]
implies that
\[
\begin{split}
1 + z_1 \bar{w}_1 \langle K_1(\cdot, \w), K_1(\cdot, \z) \rangle_{\clh_{K_1}} + & z_2 \bar{w}_2 \langle K_2(\cdot, \w), K_2(\cdot, \z) \rangle_{\clh_{K_2}}
= \vp(\z) \overline{\vp(\w)}
\\
& + \langle K_1(\cdot, \w), K_1(\cdot, \z) \rangle_{\clh_{K_1}} + \langle K_2(\cdot, \w), K_2(\cdot, \z) \rangle_{\clh_{K_2}},
\end{split}
\]
for all $\z, \w \in \D^2$. Therefore
\[
V:\begin{bmatrix} 1 \\ \bar{w}_1 K_1(\cdot, \w) \\ \bar{w}_2 K_2(\cdot, \w) \end{bmatrix}
\mapsto
\begin{bmatrix} \overline{\vp(\w)} \\ K_1(\cdot, \w) \\ K_2(\cdot, \w) \end{bmatrix} \qquad (\w \in \D^2),
\]
defines an isometry from $\cld$ onto $\clr$, where
\[
\cld = \overline{\mbox{span}}\left\{\begin{bmatrix} 1 \\ \bar{w}_1 K_1(\cdot, \w) \\ \bar{w}_2 K_2(\cdot, \w) \end{bmatrix}: \w \in\D^2\right\}
\subseteq \mathbb{C}\oplus \clh_{K_1} \oplus \clh_{K_2},
\]
and
\[
\clr = \overline{\mbox{span}}\left\{\begin{bmatrix} \overline{\vp(\w)} \\ K_1(\cdot, \w) \\ K_2(\cdot, \w) \end{bmatrix}: \w \in \D^2 \right\}
\subseteq \mathbb{C}\oplus \clh_{K_1} \oplus \clh_{K_2}.
\]
Note that
\[
\clh_{K_1} \oplus \clh_{K_2} = \overline{\mbox{span}}\left\{\begin{bmatrix} \bar{w}_1 K_1(\cdot, \w) \\ \bar{w}_2 K_2(\cdot, \w) \end{bmatrix}: \w \in \D^2 \right\}.
\]
Indeed, if
\[
\begin{bmatrix} f\\g\end{bmatrix}\in [\clh_{K_1} \oplus \clh_{K_2}]\ominus {\mbox{span}}\left\{\begin{bmatrix} \bar{w}_1 K_1(\cdot, \w) \\ \bar{w}_2 K_2(\cdot, \w) \end{bmatrix}: \w \in \D^2 \right\},
\]
then
\[
0= \Big\langle \begin{bmatrix} f\\g\end{bmatrix}, \begin{bmatrix} \bar{w}_1 K_1(\cdot, \w) \\ \bar{w}_2 K_2(\cdot, \w) \end{bmatrix} \Big\rangle_{\clh_{K_1} \oplus \clh_{K_2}}
= \langle f, \bar{w}_1  K_1(\cdot, \w) \rangle_{\clh_{K_1}} + \langle g, \bar{w}_2  K_2(\cdot, \w) \rangle_{\clh_{K_2}},
\]
that is, $w_1 f(\w) + w_2 g(\w) = 0$ for all $\w \in \D^2$. Since $K_1$ depends only on $z_1 $ and $\bar{w}_1$, all the functions in $\clh_{K_1}$ depends only on $z_1$. Therefore, if $w_2=0$, then the above equality implies that $w_1f((w_1, 0))=0$, and hence $f=0$. Consequently, $w_2g(\w)=0$, $\w \in \D^2$, and hence $g=0$, and proves our claim. In particular, $V \in \clb(\mathbb{C}\oplus \clh_{K_1} \oplus \clh_{K_2})$ is an isometry. The above proof also implies that
\[
\clh_{K_i} = \overline{\mbox{span}} \{ \bar{w}_i K_i(\cdot, \w): \w \in\D^2 \},
\]
for $i=1, 2$. Now we consider the co-isometry $V^*$ and set
\[
V^*=\begin{bmatrix}\vp(\bm{0})&B\\C&D  \end{bmatrix}
= \begin{bmatrix}\vp(\bm{0})&B_1&B_2\\C_1&D_1&D_2\\C_2&D_3&D_4  \end{bmatrix}\in
\clb(\mathbb{C}\oplus (\clh_{K_1} \oplus \clh_{K_2})).
\]
Since
\[
\begin{bmatrix}\overline{\vp(\bm{0})} & C^* \\ B^* & D^* \end{bmatrix} \begin{bmatrix} 1 \\ \bar{w}_1 K_1(\cdot, \w) \\ \bar{w}_2 K_2(\cdot, \w) \end{bmatrix}
= \begin{bmatrix} \overline{\vp(\w)} \\ K_1(\cdot, \w) \\ K_2(\cdot, \w) \end{bmatrix} \qquad (\w \in \D^2),
\]
it follows that
\[
\overline{\vp(\bm{0})} + C^*\begin{bmatrix} \bar{w}_1 K_1(\cdot, \w) \\ \bar{w}_2 K_2(\cdot, \w)
\end{bmatrix}=\overline{\vp(\w)},
\]
and
\[
B^* + D^*\begin{bmatrix}\bar{w}_1 K_1(\cdot, \w) \\ \bar{w}_2 K_2(\cdot, \w)  \end{bmatrix} = \begin{bmatrix} K_1(\cdot, \w) \\ K_2(\cdot, \w)  \end{bmatrix},
\]
for all $\w \in \D^2$. Now plug $\w = \bm{0}$ into the identity above to see that
\[
B^* = \begin{bmatrix} K_1(\cdot, \bm{0}) \\ K_2(\cdot, \bm{0}) \end{bmatrix},
\]
and hence
\[
D^*\begin{bmatrix}\bar{w}_1 K_1(\cdot, \w) \\ \bar{w}_2 K_2(\cdot, \w)  \end{bmatrix} = \begin{bmatrix} K_1(\cdot, \w) - K_1(\cdot, \bm{0}) \\ K_2(\cdot, \w) - K_2(\cdot, \bm{0}) \end{bmatrix}.
\]
Since $D^* = \begin{bmatrix}D_1^*& D_3^* \\ D_2^*& D_4^*\end{bmatrix}$, it follows that
\[
\bar{w_1} D_1^* K_1(\cdot, \w) + \bar{w}_2 D_3^* K_2(\cdot, \w) = K_1(\cdot, \w) - K_1(\cdot, \bm{0}),
\]
and
\begin{equation}\label{eqn: Agler 2nd equation}
\bar{w}_1 D_2^* K_1(\cdot, \w) + \bar{w}_2 D_4^* K_2(\cdot, \w) = K_2(\cdot, \w) - K_2(\cdot, \bm{0}).
\end{equation}
Plugging $w_2=0$ into the first identity, we get
\[
\bar{w_1} D_1^* K_1(\cdot, (w_1,0)) = K_1(\cdot, (w_1,0)) - K_1(\cdot, \bm{0}),
\]
for all $w_1 \in \D$. Again, noting that $K_1$ depends only on $z_1 $ and $\bar{w}_1$, we deduce
\[
\bar{w}_1 D_1^* K_1(\cdot, \w) = K_1(\cdot, \w) - K_1(\cdot, \bm{0}) \qquad (\w \in \D^2),
\]
and consequently $D_3^* \Big(\bar{w}_2 K_2(\cdot, \w)\Big) =0$, $\w \in \D^2$. This, along with the fact that $\{\bar{w}_2 K_2(\cdot, \w): \w \in \D^2\}$ is dense in $\clh_{K_2}$, implies $D_3 = 0$. We next plug $w_2 = 0$ into \eqref{eqn: Agler 2nd equation} to get
\[
D_2^* (\bar{w_1} K_1(\cdot, (w_1,0))) = K_2(\cdot, (w_1,0)) - K_2(\cdot, \bm{0}).
\]
Now we turn to compute $C_1^*$. Since $C^*\begin{bmatrix} \bar{w}_1K_1(\cdot, \w) \\ \bar{w} K_2(\cdot, \w) \end{bmatrix} = \overline{\vp(\w)}-\overline{\vp(\bm{0})}$, we have
\[
C_1^*(\bar{w}_1 K_1(\cdot, \w)) + C_2^* (\bar{w}_2 K_2(\cdot, \w)) = \overline{\vp(\w)} - \overline{\vp(\bm{0})} \qquad (\w \in \D^2).
\]
In particular, if $w_2 = 0$, then
\[
C_1^*(\bar{w}_1 K_1(\cdot, \w)) = \overline{\vp((w_1, 0))} - \overline{\vp(\bm{0})} \qquad (w_1 \in \D).
\]
Finally, we compute $B_2$. Observe that
\[
B_2^* + D_2^* (\bar{w}_1 K_1(\cdot, \w)) + D_4^*(\bar{w}_2 K_2(\cdot, \w)) = K_2(\cdot, \w),
\]
for all $\w \in \D^2$. If $w_2 = 0$, then
\[
B_2^* + D_2^* (\bar{w}_1 K_1(\cdot, (w_1, 0))) = K_2(\cdot, (w_1,0)),
\]
which implies that $B_2^* = K_2(\cdot, \bm{0})$. Finally, if we let $\w \in \D^2$, then
\[
B_2^*C_1^*(\bar{w}_1 K_1(\cdot, \w)) = (\overline{\vp (w_1,0)} - \overline{\vp(\bm{0})})K_2(\cdot, \bm{0}) = \overline{\vp(\bm{0})} K_2(\cdot, (w_1,0)) - \overline{\vp(\bm{0})} K_2(\cdot, \bm{0}),
\]
by assumption \eqref{eqn: agler condition 4}, and hence
\[
B_2^*C_1^*(\bar{w}_1 K_1(\cdot, \w)) = \overline{\vp(\bm{0})} (K_2(\cdot, (w_1,0)) - K_2(\cdot, \bm{0})) = \overline{\vp(\bm{0})} D_2^* (\bar{w}_1 K_1(\cdot, \w)).
\]
This proves that $\vp(\bm{0}) D_2 = C_1 B_2$.

\NI  (4)$\Rightarrow$ (1) is essentially along the lines of \cite[Theorem 2.3]{DS}. However, for the sake of completeness, we sketch the proof. Let $a=\vp(\bm{0})$. Since $V V^* = I$, it follows that
\[
I = \begin{bmatrix}|a|^2+B_1B_1^*+B_2B_2^*&aC_1^*+B_1D_1^*+B_2D_2^*&aC_2^*+B_2D_4^*
\\
\overline{a}C_1+D_1B_1^*+D_2B_2^* &C_1C_1^*+D_1D_1^*+D_2D_2^*&C_1C_2^*+D_2D_4^*
\\
\overline{a}C_2+D_4B_2^* & C_2C_1^*+D_4D_2^* & C_2C_2^*+D_4D_4^*
\end{bmatrix}.
\]
Then there exists $y\in \mathbb{C}$ such that
\[
|y|^2=|a|^2+B_2B_2^*=1-B_1B_1^*>0,
\]
as $a \neq 0$. Let $x=\frac{a}{y}$, and
\[
V_1 = \begin{bmatrix} y & B_1
\\
\frac{1}{x}C_1 & D_1 \end{bmatrix} \quad \text{and} \quad
V_2 = \begin{bmatrix} x & \frac{1}{y}B_2 \\
C_2 & D_4\end{bmatrix}.
\]
Clearly, $x \neq 0$. We first claim that $V_1$ and $V_2$ are co-isometries. Indeed
\[
V_2 V_2^* = \begin{bmatrix}|x|^2+\frac{1}{|y|^2}B_2B_2^*&xC_2^*+\frac{1}{y}B_2D_4^*
\\
\bar{x}C_2+\frac{1}{\bar{y}}D_4B_2^* & C_2C_2^*+D_4D_4^*
\end{bmatrix}  = \begin{bmatrix} 1 & xC_2^*+\frac{1}{y}B_2D_4^*
\\
\bar{x}C_2+\frac{1}{\bar{y}}D_4B_2^* & C_2C_2^*+D_4D_4^*
\end{bmatrix}.
\]
as $|y|^2=|a|^2+B_2B_2^*$ and $a=xy$. Also note that, since $aC_2^*+B_2D_4^*=0$, we have that $xC_2^*+\frac{1}{y}B_2D_4^*=0$, which implies that $V_2$ is a co-isometry. Next, we compute
\[
V_1V_1^* = \begin{bmatrix} |y|^2+B_1B_1^* & \frac{y}{\overline{x}}C_1^* + B_1D_1^*
\\
\frac{\overline{y}}{x}C_1 + D_1B_1^*&\frac{1}{|x|^2}C_1C_1^*+D_1D_1^*\end{bmatrix}.
\]
Since $C_1C_1^*+D_1D_1^*+D_2D_2^*=1$, $aD_2=C_1B_2$, $a=xy$ and $|y|^2-|a|^2=B_2B_2^*$, we have
\[
\frac{1}{|x|^2}C_1C_1^*+D_1D_1^*=1.
\]
Moreover, since $aC_1^*+B_1D_1^*+B_2D_2^*=0$ implies that $\frac{y}{\overline{x}}C_1^*+B_1D_1^*=0$, we have that $V_1$ is also a co-isometry. Finally, set $\vp_1(\z)=\tau_{V_1}(z_1)$ and $\vp_2(\z) = \tau_{V_2}(z_2)$, $\z \in \D^2$. It is then easy to check that
\[
\vp(\z) = \tau_V(z) = \tau_{V_1}(z_1) \tau_{V_2}(z_2) = \vp_1(\z) \vp_2(\z),
\]
for all $\z \in \D^2$. This completes the proof.
\end{proof}

%\begin{rem}\label{rem: non-trivial factors}
%Suppose $\vp \in \cls(\D^2)$ satisfies the hypotheses of part (1) of Theorem \ref{thm: agler kernel and fact}. That is, $\vp(\z) = \vp_1(z_1) \vp_2(z_2)$, $\z \in \D^2$, for some $\vp_1, \vp_2 \in \cls(\D)$. Let $\vp_i = \tau_{V_i}$, $i=1,2$, for some isometric colligation
%\[
%V_i = \begin{bmatrix} a_i & B_i \\ C_i & D_i \end{bmatrix} \in \clb(\mathbb{C} \oplus \clh_i).
%%Then it is easy to see that $\vp_1 $ and $\vp_2$ are non-constant if and only if
%\[
%B_1 D^{p}_1 C_1 \neq 0 \quad \mbox{and} \quad B_2 D^{q}_2 C_2 \neq 0,
%\]
%for some $p, q \geq 0$. This is also equivalent to: The Agler kernels $\{K_1, K_2\}$ of $\vp$ satisfies the additional conditions that $(1-z_1\bar{w}_1) K_1(\z, \w)$ and $(1-z_2 \bar{w}_2) K_2((0,z_2), (0, w_2))$ are non-constant functions. This conclusion results from the following general fact: If $\psi \in \cls(\D)$, then $\psi$ is non-constant if and only if
%\[
%(z, w) \mapsto  (1 - \psi(z) \overline{\psi(w)}),
%\]
%is non-constant. Therefore, the above conditions ensure non-trivial factorizations of $\vp$ in Theorem \ref{thm: agler kernel and fact}.
%\end{rem}

In the setting of Theorem \ref{thm: agler kernel and fact}, one can also explicitly compute the entries of the block operator matrix $V$ in part (4). The technique involved in the computation is standard and quite well known (cf. \cite[Remark 3.6]{BB}). However, we outline some details for the sake of making this paper self-contained. We already know that
\[
B_2^* = K_2(\cdot, \bm{0}) \quad \mbox{and} \quad C_1^*(\bar{w}_1 K_1(\cdot, \w)) = \overline{\vp((w_1, 0))} - \overline{\vp(\bm{0})},
\]
and
\[
D_2^* (\bar{w_1} K_1(\cdot, (w_1,0))) = K_2(\cdot, (w_1,0)) - K_2(\cdot, \bm{0}),
\]
for all $\w \in \D^2$. Now let $g\in \clh_{K_2}$ and $\w \in \D^2$. Then
\[
(z_1 D_2 g)(\w) = \langle g, K_2(\cdot, (w_1, 0)) - K_2(\cdot, \bm{0})\rangle = g((w_1, 0))- g (\bm{0}),
\]
and hence
\[
(D_2 g)(\w) = \frac{g((w_1, 0)) - g(\bm{0})}{w_1} \qquad (\w \in \D^2).
\]
for all $g \in \clh_{K_2}$. Similarly, if $w_1 = 0$, then \eqref{eqn: Agler 2nd equation} implies that
\[
\bar{w}_2 D_4^* K_2(\cdot, \w) = K_2(\cdot, \w) - K_2(\cdot, (w_1, 0)),
\]
and hence, in a similar way we have
\[
(D_4 g)(\w) = \frac{g(\w) - g((w_1, 0))}{w_2} \qquad (g \in \clh_{K_2}, \w \in \D^2),
\]
as well as
\[
(D_1 f)(\w) = \frac{f(\w) - f(\bm{0})}{w_1} \qquad (f\in \clh_{K_1}, \w \in \D^2).
\]
Now we turn to compute $C_1$ and $C_2$. Since $ C_1^*(\bar{w}_1 K_1(\cdot, \w)) = \overline{\vp((w_1, 0))} - \overline{\vp(\bm{0})}$, we have
\[
(z_1 C_1 1)(\w) = \langle C_1 1, \bar{w}_1 K_1(\cdot, \w) \rangle =  {\vp((w_1, 0))} - {\vp(\bm{0})},
\]
and hence
\[
(C_1 1)(\w) = \frac{\vp(w_1,0)-\vp(\bm{0})}{w_1}\quad \mbox{and}\quad
(C_2 1)(\w)=\frac{\vp(\w)-\vp(w_1,0)}{w_2},
\]
for all $\w \in \D^2$. Finally, we note that $(B_1 f)(\w) = f(\bm{0})$ and $(B_2 g)(\w) = g(\bm{0})$ for all $f\in \clh_{K_1}$ and $g\in \clh_{K_2}$.

In particular, if $\vp$ is inner, then we have the following:

\begin{ex}
Given an inner function $\vp \in \cls(\D^2)$ satisfying one of the equivalent conditions of Theorem \ref{thm: agler kernel and fact}, we have $\vp(\z) = \vp_1(z_1) \vp_2(z_2)$, $\z \in \D^2$, for some $\vp_1$ and $\vp_2$ in $\cls(\D)$. Then
\[
1=|\vp(\z)| = |\vp_1(z_1)| |\vp_2(z_2)| \leq |\vp_1(z_1)| \leq 1 \qquad (\z \in \mathbb{T}^2 \mbox{ a.e.})
\]
from which we see that $\vp_1$, as well as $\vp_2$, are inner functions. Moreover, for $\z , \w \in \D^2$, we have
\[
1-\vp(\z) \overline{\vp(\w)} = 1 - \vp_1(z_1) \overline{\vp_1(w_1)} + \vp_1(z_1) (1-\vp_2(z_2) \overline{\vp_2(w_2)})\overline{\vp_1(w_1)}.
\]
Hence $\{K_1, K_2\}$ are Agler kernels of $\vp$, where
\[
K_1(\z, \w) = \frac{1-\vp_1(z_1) \overline{\vp_2(w_1)}}{1-z_1\bar{w}_1} \quad \mbox{and} \quad
K_2(\z, \w) = \frac{\vp_1(z_1) (1 - \vp_2(z_2)\overline{\vp_2(w_2)})\overline{\vp_1(w_1)}}{1 - z_2 \bar{w}_2}.
\]
In this case the corresponding reproducing kernel Hilbert spaces are given by
\[
\clh_{K_1} = \clq_{\vp_1} \otimes \mathbb{C} \quad \text{and}\quad \clh_{K_2} = \vp_1 \mathbb{C} \otimes \clq_{\vp_2},
\]
where $\clq_{\vp_1} = H^2(\D)/ \vp_1 H^2(\D)$ and $\clq_{\vp_2} = H^2(\D)/ \vp_2 H^2(\D)$ are model spaces. Moreover, the co-isometric (unitary) colligation operator $V$ with state space $\clh_{K_1} \oplus \clh_{K_2}$ is given by
\[
V=\begin{bmatrix}\vp(\bm{0}) & P_{\mathbb{C}}|_{\clq_{\vp_1}} & \vp_1(0) P_{\mathbb{C}} M^*_{\vp_1} \otimes P_{\mathbb{C}}|_{\clq_{\vp_2}}
\\
\vp_2(0) M_z^* M_{\vp_1}|_{\mathbb{C}} & M_z^*|_{\clq_{\vp_1}} & M_z^* M_{\vp_1} P_{\mathbb{C}} M^*_{\vp_1} \otimes P_{\mathbb{C}}|_{\clq_{\vp_2}}
\\
M_{\vp_1}|_{\mathbb{C}} \otimes M_z^*M_{\vp_2}|_{\mathbb{C}} & 0 & I_{\vp_1 \mathbb{C}} \otimes M_z^*|_{\clq_{\vp_2}}
\end{bmatrix}.
\]
\end{ex}

\smallskip

Finally, we comment on the assumption that $\vp(\bm{0}) \neq 0$ in Theorem \ref{thm: agler kernel and fact}.

\begin{rem}\label{rem: phi(0) =0}
In the proof of Theorem \ref{thm: agler kernel and fact}, $\vp(\bm{0}) \neq 0$ has been used only for the implication (4)$\Rightarrow$ (1). In the $\vp(\bm{0}) = 0$ case, one can easily modify the argument of the aforementioned case to prove a similar statement. Here is a sample statement:

\NI Let $\vp \in \cls(\D^2)$ be a non-zero function and suppose $\vp(\bm{0}) = 0$. Then the following are equivalent:

(1) $\vp(\z) = \vp_1(z_1) \vp_2(z_2)$ for some $\vp_1, \vp_2 \in \cls(\D)$ such that $\vp_2(0) \neq 0$.

(2) $\vp(\z) = z_1^{p} \vp_1(z_1) \vp_2 (z_2)$ for some $p \geq 1$ and $\vp_1, \vp_2  \in \cls(\D)$ such that $\vp_1(0)\neq 0$ and $\vp_2(0) \neq 0$.

(3) There exists $p \geq 1$ such that $\tilde \vp(\z) = z_1^{-p}\vp(\z) \in \cls(\D^2)$, $\tilde\vp(\bm{0})\neq 0$, and there exist Agler kernels $\{K_1,K_2\}$ of $\tilde \vp$ such that	$K_1$ depends only on $z_1$ and $\bar{w}_1$, and
\[
\overline{\tilde\vp(\bm{0})} K_2(\cdot, (w_1,0)) = \overline{\tilde\vp (w_1,0)} K_2(\cdot, \bm{0}) \qquad (w_1 \in \D).
\]

(4) There exists $p \geq 1$ such that $\tilde \vp(\z) = z_1^{-p} \vp(\z) \in \cls(\D^2)$, $\tilde \vp(\bm{0})\neq 0$, and $\tilde \vp = \tau_V$ for some co-isometric colligation
\[
V=\begin{bmatrix} \tilde \vp(\bm{0})&B_1&B_2\\C_1&D_1&D_2\\C_2&0&D_4 \end{bmatrix},
\]
such that $\tilde\vp(\bm{0}) D_2 = C_1 B_2$.
\end{rem}

\section{Counterexamples and a converse}\label{sec: Counterexamples}

We now return to two-variable inner functions, which we encountered in Section \ref{sec-4}. The aim of this section is to further analyze Theorem \ref{thm-inner new}. We begin by exhibiting counterexamples to the converse of Theorem \ref{thm-inner new}. Then, in Theorem \ref{thm: weak converse}, we present a weak converse to Theorem \ref{thm-inner new}.

\begin{ex}\label{example: counter}
Fix $t \in (0,1)$, and define
\[
\vp_t(\z)=\frac{z_1z_2-t}{1-tz_1z_2} \qquad (\z\in\D^2).
\]
It is fairly easy to verify that
\[
|\vp_t(\z)| = 1 \qquad (\z \in \mathbb{T}^2),
\]
and hence, $\vp_t$ is a rational inner function. Contrary to what we want, let us assume that there are Hilbert spaces $\clh_1$ and $\clh_2$, an operator $D_1 \in C_{0 \cdot}$, and an isometric colligation
\[
V_t=\begin{bmatrix} -t& B_1&B_2\\C_1&D_1&D_2\\C_2&0&D_3 \end{bmatrix}
\in \clb(\mathbb{C}\oplus\clh_1\oplus\clh_2)
\]
such that $\tau_{V_t} =\vp_t$. Since
\[
\vp_t(\z) + t = \frac{(1-t^2)z_1z_2}{1-tz_1z_2},
\]
the preceding equality yields
\[
\frac{(1-t^2)z_1z_2}{1-tz_1z_2}= \begin{bmatrix}B_1 & B_2 \end{bmatrix} \left(
\begin{bmatrix} I&0\\0&I \end{bmatrix}-  \begin{bmatrix} z_1&0\\0&z_2 \end{bmatrix}
\begin{bmatrix} D_1&D_2\\0&D_3 \end{bmatrix}
\right)^{-1} \begin{bmatrix} z_1&0\\0&z_2 \end{bmatrix} \begin{bmatrix} C_1\\C_2 \end{bmatrix}.
\]
Now the left side is equal to
\[
(1-t^2)z_1z_2(1+tz_1z_2+t^2z^2_1z^2_2+\cdots),
\]
and the right side is equal to
\[
z_1B_1(I-z_1D_1)^{-1}C_1 + z_2B_2(I-z_2D_4)^{-1}C_2 + z_1z_2B_1(I-z_1D_1)^{-1} D_2(I-z_2 D_3)^{-1}C_2.
\]
Comparing the coefficients of $z_1$, we see that $B_1D_1^nC_1=0$, $n\geq 0$. Since $V_t^* V_t = I$, we have
\[
\begin{bmatrix} -t& C^*_1&C^*_2\\B^*_1&D^*_1&0\\B^*_2&D_2^*&D^*_3 \end{bmatrix}
\begin{bmatrix} -t& B_1&B_2\\C_1&D_1&D_2\\C_2&0&D_3 \end{bmatrix} = \begin{bmatrix} 1& 0&0\\0&I&0\\0&0&I \end{bmatrix}.
\]
In particular $B_1^*B_1+D_1^*D_1=I$ and $-tB_1+C_1^*D_1=0$. The first equality implies (see the proof of the equality in \eqref{eqn: infinte sum 1 =1}) that
\[
\sum_{n=0}^\infty D_1^{*n} B_1^*B_1 D_1^n = I,
\]
in the strong operator topology as $D_1 \in C_{0 \cdot}$. Therefore
\[
\sum_{n=0}^{\infty} \|B_1D_1^nh \|^2 = \|h\|^2,
\]
for all $h \in\clh_1$. In particular, if we choose $h = C_1(1)$, then
\[
\sum_{n=0}^{\infty} \|B_1D_1^nC_1(1) \|^2 = \|C_1(1)\|^2.
\]
Since $B_1D_1^nC_1=0$ for all $n\geq 0$, we deduce$C_1=0$. Then $-tB_1+C_1^*D_1=0$ implies that $B_1=0$, and hence $D_1^*D_1=I$. However, this and the fact that $D_1 \in C_{0 \cdot}$ are mutually contradictory. This shows that $\vp_t \neq \tau_{V_t}$ for any isometric colligation $V_t$ and $D_1 \in C_{0 \cdot}$.
\end{ex}

Now we turn to a weak converse of Theorem \ref{thm-inner new} in the setting of rational inner functions. We need the following inverse formula of $2 \times 2$ block matrices \cite[page 18]{HJ}:

\begin{thm}
Let $X=\begin{bmatrix}P&Q\\R&S \end{bmatrix}\in \clb(\mathbb{C}^m\oplus \mathbb{C}^n)$, and suppose that $P$ is invertible. Then $X$ is invertible if and only if $\Delta:= S-RP^{-1}Q$ is invertible. In this case, the inverse of $X$ is given by
\[
X^{-1} = \begin{bmatrix}P^{-1}+P^{-1}Q \Delta^{-1} R P^{-1} & -P^{-1} Q \Delta^{-1}
\\
- \Delta^{-1} R P^{-1} & \Delta^{-1}
\end{bmatrix}.
\]
\end{thm}

We are now ready to establish the promised weak converse of Theorem \ref{thm-inner new}.

\begin{thm}\label{thm: weak converse}
Let $\vp \in \cls(\D^2)$ be a rational inner function and suppose $\vp(\bm{0})\neq 0$. Then the following are equivalent:

(1) $\vp = \tau_V$ for some isometric colligation
\[
V=\begin{bmatrix}a & B\\C&D\end{bmatrix}=
\begin{bmatrix}
a & B_1&B_2
\\
C_1&D_1&D_2
\\
C_2 & 0 & D_3\end{bmatrix} \in \clb(\mathbb{C}\oplus\clh_1\oplus\clh_2),
\]
where $\clh_1$ and $\clh_2$ are finite-dimensional Hilbert spaces and $D_1, D_3\in C_{0 \cdot }$.

(2) $\vp(\z)=\vp_1(z_1)\vp(z_2)$, $\z \in \D^2$, for some rational inner functions $\vp_1$ and $\vp_2$ (in $\cls(\D)$).
\end{thm}

\begin{proof}
(1)$\Rightarrow$ (2): Since $V \in \clb(\mathbb{C}\oplus\clh_1\oplus\clh_2)$ is an isometry and $\mbox{dim} \clh_i < \infty$, $i=1,2$, $V$ is onto, that is, $V$ is a unitary operator. In particular, $V$ is invertible. Since
\[
a=\vp(\bm{0})\neq 0,
\]
by the above theorem, we conclude that $aD-CB$ is invertible and
\[
V^{-1}=
\begin{bmatrix} a^{-1}+a^{-1}B(aD-CB)^{-1}C& -B(aD-CB)^{-1}\\
-(aD-CB)^{-1}C& a^{-1}(aD-CB)^{-1}
\end{bmatrix}.
\]
Since $V^*=V^{-1}$, in particular, we have
\[
D^* =
\begin{bmatrix} D_1^*&0\\D_2^*& D_3^* \end{bmatrix}=a^{-1}(aD-CB)^{-1}=
a^{-1}\begin{bmatrix} aD_1-C_1B_1 & aD_2-C_1B_2 \\ -C_2B_1 & aD_3-C_2B_2 \end{bmatrix}^{-1},
\]
and hence
\[
a\begin{bmatrix} aD_1-C_1B_1 & aD_2-C_1B_2 \\ -C_2B_1 & aD_3-C_2B_2 \end{bmatrix}\begin{bmatrix} D_1^*&0\\D_2^*&D_3^* \end{bmatrix}=\begin{bmatrix} I&0\\0&I \end{bmatrix}.
\]
But then this implies $(aD_2-C_1B_2)D_3^*=0$. Note that the invertibility of $D$ immediately implies that $D_3$ is also invertible. Then $aD_2-C_1B_2=0$,
and hence, by Theorem \ref{thm: agler kernel and fact}, there exist rational inner functions $\vp_1$ and $\vp_2$ (here $\clh_1$ and $\clh_2$ are finite dimensional Hilbert spaces) such that $\vp(\z) = \vp_1(z_1)\vp(z_2)$, $\z \in \D^2$.

\NI (2)$\Rightarrow$ (1): Since $\vp_i$ ($\in \cls(\D)$) is a rational inner function, there exists an isometric colligation
\[
V_i = \begin{bmatrix}
a_i&B_i\\C_i&D_i
\end{bmatrix} \in \clb(\mathbb{C}\oplus \clh_i),
\]
such that $\mbox{dim}(\clh_i)<\infty$, $D_i\in C_{0 \cdot }$, and $\vp_i = \tau_{V_i}$ for all $i=1,2$. We define
\[
V=\begin{bmatrix} a_1a_2&B_1&a_1B_2\\a_2C_1&D_1&C_1B_2\\C_2&0&D_2 \end{bmatrix}.
\]
Then a somewhat careful computation (or see the proof of \cite[Theorem 2.2]{DS}) yields that $\vp = \tau_V$.
\end{proof}

In this connection, and also in the context of Remark \ref{rem: D C dot zero}, it is probably worth mentioning that in the finite dimensional case we have the following: If $\begin{bmatrix}D_1&D_2\\0&D_3 \end{bmatrix} \in \clb(\mathbb{C}^p \oplus \mathbb{C}^q)$ for some $p, q \geq 1$, then
\[
\sigma\Big(\begin{bmatrix}D_1&D_2\\0&D_3 \end{bmatrix} \Big) =\sigma(D_1)\cup\sigma(D_3),
\]
and in particular, $\begin{bmatrix}D_1&D_2\\0&D_3 \end{bmatrix} \in C_{0 \cdot}$ if and only if $D_1, D_3 \in C_{0 \cdot}$.

Finally, we point out that part $(1)$ of Theorem \ref{thm: agler kernel and fact} and part $(2)$ of Theorem \ref{thm: weak converse} are related (in a different direction) to essential normality of Beurling type quotient modules of $H^2(\D^2)$ \cite{GW}.

\vspace{0.1in}

\noindent{Acknowledgement:} The second author is supported in part by NBHM (National Board
of Higher Mathematics, India) grant NBHM/R.P.64/2014, and the Mathematical Research
Impact Centric Support (MATRICS) grant, File No : MTR/2017/000522, by the Science and
Engineering Research Board (SERB), Department of Science \& Technology (DST), Government of India.

\end{document}